\documentclass{article}

\usepackage[
  journal=JNCG,
  lang=british,
]{ems-journal}

\makeatletter
\renewcommand*\ps@titlepage{%
  \let\@oddfoot\@empty
  \let\@evenfoot\@empty
  \let\@oddhead\@empty
  \let\@evenhead\@empty
}
\makeatother

\usepackage{tikz-cd}
\usetikzlibrary{matrix}

\theoremstyle{plain}

\newtheorem{theorem}[subsubsection]{Theorem}

\newtheorem{lemma}[subsubsection]{Lemma}
\newtheorem{proposition}[subsubsection]{Proposition}

\theoremstyle{definition}

\newtheorem{definition}[subsubsection]{Definition}

\theoremstyle{remark}

\newtheorem{remark}[subsubsection]{Remark}

\numberwithin{equation}{section}

\newcommand{\QQ}{\mathbb{Q}}
\newcommand{\ZZ}{\mathbb{Z}}
\newcommand{\PP}{\mathbb{P}}

\newcommand{\HoH}{\mathrm{HH}}
\newcommand{\bl}{\mathrm{bl}}

\newcommand{\rk}{\mathrm{rk}}

\newcommand{\Hom}{\mathrm{Hom}}
\newcommand{\Ext}{\mathrm{Ext}}

\newcommand{\ch}{\mathrm{ch}}

\newcommand{\pr}{\mathrm{pr}}

\newcommand{\cB}{\mathcal{B}}

\newcommand{\cK}{\mathcal{K}}

\newcommand{\cP}{\mathcal{P}}

\makeatletter
\@ifl@t@r\fmtversion{2024-11-01}{%
  \def\H@refstepcounter#1{%
    \stepcounter{#1}%
    \edef\@currentcounter{#1}%
    \protected@edef\@currentlabel{%
      \csname p@#1\expandafter\endcsname\csname the#1\endcsname}%
  }%
}{}
\makeatother

\newcommand{\Perf}{\mathrm{Perf}}
\newcommand{\RHom}{\mathbf{R}\!\Hom}
\newcommand{\End}{\mathrm{End}}
\newcommand{\rad}{\mathrm{rad}}

\newcommand{\NHH}{\mathrm{NHH}}
\newcommand{\wt}{\mathrm{wt}}

\newcommand{\tensor}{\otimes}

\begin{document}

\title{Phantom categories on 3-vertex directed DG quivers}

\emsauthor{2}{
  \givenname{Yeqin}
  \surname{Liu}
  \mrid{1591694}
  \orcid{0000-0001-8231-7682}}{Y.~Liu}

\Emsaffil{2}{
  \department{Department of Mathematics}
  \organisation{University of Michigan}
  \rorid{00jmfr291}
  \address{530 Church St}
  \zip{48109}
  \city{Ann Arbor, MI}
  \country{USA}
  \affemail{yqnl@umich.edu}
}

\emsauthor{2}{
	\givenname{Yu}
	\surname{Shen}
	\mrid{}
	\orcid{0000-0001-5766-1596}}{Y.~Shen}

\Emsaffil{2}{
  \department{Department of Mathematics}
  \organisation{Florida State University}
  \rorid{05g3dte14}
  \address{1017 Academic Way}
  \zip{32306}
  \city{Tallahassee, FL}
  \country{USA}
  \affemail{ys26k@fsu.edu}
}

\classification[18G80, 16E40]{16E45}

\begin{abstract}
We construct infinitely many pairwise non-derived-Morita-equivalent phantom
categories $\cP_{m}, m\in\ZZ$ from concrete directed DG quivers with 3 vertices.
Moreover, there exists a fixed finite-dimensional DG algebra $C$ such that for each $m\in\ZZ$, there is a way to reassign cohomological gradings on $C$ (without changing the differential) to obtain a new DG algebra $C_{m}$ with $\Perf(C_{m})\cong \cP_{m}$.
\end{abstract}

\maketitle

\section{Introduction}\label{section-intro}
We work over a field $k$ of characteristic 0.
A nonzero smooth proper DG category $\cP$ is called a \emph{phantom category} if $K_{0}(\cP)=0$ and $\HoH_{\bullet}(\cP)=0$. The first examples of phantom categories were constructed in \cite{GO} and independently in \cite{BBS}, arising from
derived categories of surfaces of general type. Further examples on Dolgachev surfaces were constructed in \cite{ChoLeeDolgachev} and \cite{KarzhemanovKatzarkov}. Later, \cite{Krah} constructed a universal phantom on $\bl_{10}\PP^{2}$. Recent constructions on rational surfaces include \cite{LoomingPhantoms}, \cite{MaXiongYangNew,MaXiongYangEchoes}, and \cite{HeLiuShengZhou}.

The purpose of this paper is to construct infinitely many pairwise
non-derived-Morita-equivalent phantoms from a single finite-dimensional algebra. The beautiful constructions in the above literature are mostly abstract or on a relatively large category, which makes concrete computations difficult. Our motivation is to produce phantom categories from small, concrete quivers whose algebras and modules are explicit enough to be studied directly. Such an example would be helpful in understanding internal structures of phantom categories, such as t-structures, moduli theory of objects, and more. Eventually, we hope that this program leads to a mass production of concrete smooth and finite-dimensional DG algebras examples.

Our starting point is the observation that a directed 3-vertex quiver in \cite{BondalPolishchuk, KuznetsovJH, SungBondalQuiver} admits an ``unexpected'' exceptional object, distinct from the standard ones. There, this phenomenon does not continue to a second layer, but we believe that this is a failure of that particular ambient quiver, rather than an essential obstruction. This suggests choosing the ambient quiver more carefully so that the same phenomenon occurs for all three steps, producing an unexpected exceptional collection of length $3$.

Motivated by this observation, we construct the algebra $A$ defined in
(\ref{equation-def-A}) as a quotient of the path algebra of a
3-vertex directed quiver by relations. In $\Perf(A)$, there is an ``unexpected'' length-3 exceptional collection $\langle E,F,G \rangle$ defined in subsection \ref{section-modules}, different from the standard one. For each $m\in\ZZ$, there is a way to reassign cohomological gradings of $A$ and $\langle E,F,G \rangle$ in a compatible way. Write $A_{m}$ and $\langle E_{m},F_{m},G_{m}\rangle$ for the corresponding regradings. Then $\cP_{m}=\langle E_{m},F_{m},G_{m}\rangle^{\perp}\subset \Perf(A_{m})$ are pairwise distinct phantom categories. More specifically, we have the following main theorem.

\begin{theorem}\label{theorem-main}
For every $m\in\ZZ$, $\cP_{m}=\langle E_{m}, F_{m}, G_{m} \rangle^{\perp}$ is a phantom category. The Hochschild cohomology satisfies:
\begin{itemize}
\item If $m>0$, we have $\HoH^{j}(\cP_{m})=0$ for $j\notin[2-8m,5+6m]$, and
$$
\HoH^{2-8m}(\cP_{m})\cong k, 
\qquad
\HoH^{5+6m}(\cP_{m})\cong k^{2}.
$$
\item If $m=0$, then $\HoH^{j}(\cP_{0})=0$ for $j\notin[0,5]$, and
$$
\HoH^{0}(\cP_{0})\cong k,
\qquad
\HoH^{5}(\cP_{0})\cong k^{140}.
$$
\item If $m<0$, we have $\HoH^{j}(\cP_{m})=0$ for $j\notin[2+6m,5-8m]$, and
$$
\HoH^{2+6m}(\cP_{m})\cong k,
\qquad
\HoH^{5-8m}(\cP_{m})\cong k.
$$
\end{itemize}
In particular, the categories $\cP_{m}$ are pairwise non-derived-Morita-equivalent.
\end{theorem}

Note that the DG algebras $A_{m}$ all have the same underlying algebra. This is not a common feature that works for arbitrary DG algebras: it imposes constraints on both the ambient quiver and the exceptional collection. The particular algebra and modules in this paper are not the limit of this method. In principle the same construction can be applied to other carefully chosen internally graded (DG) quivers with unexpected exceptional collections. This provides a general strategy for producing more families of residual categories.

\subsection*{Acknowledgments}

Much credit should be given to Alexander Kuznetsov:
We learned that he had shown that Krah's phantom \cite{Krah} can be realized on an ambient category with a full exceptional collection of length $3$.
We thank Amal Mattoo and Alexander Perry for independently bringing this
fact to our attention. This gave us confidence to look for phantom
residuals on concrete 3-vertex directed DG quivers (instead of higher lengths).

\subsection*{AI disclosure}

This project began in November 2025. A substantial portion of the work was
carried out by the authors without the use of AI. The mathematical ideas are designed by the authors.  We used AI tools (including
ChatGPT 5.6 Sol and ChatGPT 6 Astra) to assist with numerical ansatz, computations, language editing and grammatical corrections. The final paper presentation is carefully written by the authors. The authors
take full responsibility for the correctness of the paper.

\section{Construction of algebra and exceptional modules}\label{section-fixed-data}

In this section, we construct the path algebra $A$ of a 3-vertex directed quiver with relations, together with
an exceptional triple $(E,F,G)$ that is distinct from the standard one. 

We equip both the algebra and the modules
with an additional $\ZZ$-grading (distinct from the cohomological grading), which we call the \emph{internal grading}.
For a homogeneous element $x$, we write $\wt(x)$ for its internal degree.
For a finite internally graded vector space
$
M=\bigoplus_{w} M_{w},
$
we define its internal graded character by
$$
\ch(M)=\sum_{w} \dim_{k}(M_{w})q^{w}.
$$
The internal character is an enrichment of $\dim_{k}(M)$: we have $\ch(M)|_{q=1}=\dim_{k}(M)$.

\subsection{The quiver algebra}\label{section-algebra}
First, we construct the algebra $A$ together with an internal
$\ZZ$-grading from a 3-vertex quiver with homogeneous quadratic relations in this subsection.

Consider the quiver $Q$:
$$
1 \overset{a_{1},\ldots,a_{4}}{\longrightarrow} 2
  \overset{b_{1},\ldots,b_{5}}{\longrightarrow} 3.
$$
We write $b_{j}a_{i}$ for the length 2 path obtained by first applying $a_{i}$
and then $b_{j}$. Now equip $kQ$ with the internal grading by
$$
\wt(a_{1},a_{2},a_{3},a_{4})=(0,0,1,-1),
\qquad
\wt(b_{1},b_{2},b_{3},b_{4},b_{5})=(0,0,-1,-1,-2).
$$
We define an ideal $(\rho_{1}, \cdots, \rho_{7})\subset kQ$ by the following generators
$$
\begin{array}{rl}
\rho_{1}&=-b_{1}a_{1}+2b_{1}a_{2}+b_{2}a_{1}-2b_{2}a_{2}+2b_{3}a_{3}+2b_{4}a_{3},\\
\rho_{2}&=2b_{1}a_{2}+b_{2}a_{1}-3b_{2}a_{2}-2b_{3}a_{3},\\
\rho_{3}&=-b_{1}a_{3}+2b_{2}a_{3},\\
\rho_{4}&=-2b_{1}a_{4}-2b_{3}a_{1}+b_{3}a_{2}+b_{4}a_{2}-2b_{5}a_{3},\\
\rho_{5}&=b_{1}a_{4}-b_{2}a_{4}-2b_{3}a_{1}+b_{3}a_{2}-2b_{5}a_{3},\\
\rho_{6}&=b_{1}a_{4}-2b_{3}a_{1}-2b_{4}a_{1}+b_{4}a_{2},\\
\rho_{7}&=2b_{3}a_{4}-b_{4}a_{4}+b_{5}a_{1}+b_{5}a_{2}.
\end{array}
$$
Note that $\wt(\rho_{1},\cdots, \rho_{7})=(0,0,1,-1,-1,-1,-2)$, so $\rho_{i}$ are (internal graded) homogeneous.
Throughout this paper, define the internal graded algebra
\begin{equation}\label{equation-def-A}
    A=kQ/(\rho_{1},\ldots,\rho_{7}).
\end{equation}
In the space of length-2 paths spanned by $
(b_{1}a_{1},\ldots,b_{1}a_{4},b_{2}a_{1},\ldots,b_{5}a_{4})
$, the coefficient matrix of
$\rho_{1},\ldots,\rho_{7}$ has rank $7$. Let $J=\rad(A)$ be the Jacobson radical of $A$.  Then we have
$$
\dim_{k} A=\dim_{k}(A/J)+\dim_{k}(J/J^{2})+\dim_{k}(J^{2})=3+(4+5)+(4\cdot 5-7)=25.
$$
In this paper, we fix the following vector spaces:
\begin{equation}\label{equation-UVW}
\begin{array}{rl}
    U=\langle a_{1},\ldots,a_{4}\rangle, 
    \quad &
    V=\langle b_{1},\ldots,b_{5}\rangle,
\\
    R=\langle \rho_{1},\ldots,\rho_{7}\rangle\subset V\tensor U, 
    \quad &
    W=(V\tensor U)/R.
\end{array}
\end{equation}
Their internal graded characters are
\begin{align*}
\ch(U)&=2+q+q^{-1},\\
\ch(V)&=2+2q^{-1}+q^{-2},\\
\ch(R)&=2+q+3q^{-1}+q^{-2},\\
\ch(W)&=\ch(U)\ch(V)-\ch(R)
     =q+4+4q^{-1}+3q^{-2}+q^{-3}.
\end{align*}

\subsection{The graded modules $E,F,G$}\label{section-modules}

The goal of this subsection is to construct 3 explicit internally graded $A$-modules $E,F,G$, which will be shown in the next subsection to form the unexpected exceptional triple used throughout the paper.

Recall that
a left $A$-module $X$ is a representation
$$
X_{1}\overset{A_{1}^{X},\ldots,A_{4}^{X}}{\longrightarrow}
X_{2}\overset{B_{1}^{X},\ldots,B_{5}^{X}}{\longrightarrow}X_{3},
$$
where $A_{i}^{X}:X_{1}\to X_{2}$ and $B_{j}^{X}:X_{2}\to X_{3}$ are the actions of
$a_{i}$ and $b_{j}$, respectively. For a representation $X$, recall that the dimension vector
$$
\overrightarrow{\dim}(X)=\bigl(\dim_{k} X_{1},\dim_{k} X_{2},\dim_{k} X_{3}\bigr).
$$
For our construction, take
$$
\overrightarrow{\dim}(E)=(3,2,1),\qquad
\overrightarrow{\dim}(F)=(1,1,1),\qquad
\overrightarrow{\dim}(G)=(1,2,4).
$$
These choices are made so that their numerical classes in $K_{0}$ have numerical semiorthogonality. This can be made explicit by (\ref{equation-euler-form}), and for our particular construction this is checked in Remark \ref{remark-upper-triangular}. After the dimension vectors are fixed, we will choose the internal degrees and the matrices below. They are chosen so that the seven quadratic relations are homogeneous and satisfied, while the resulting $\Ext$ complexes have the rank conditions needed for making the triple an exceptional collection. This entire construction is one explicit solution to these constraints, rather than any canonical choice.

We next choose homogeneous bases for the three modules, with respect to the internal grading. The internal degrees
of basis vectors at the three vertices are given by
$$
\begin{array}{c|ccc}
 & X_{1} & X_{2} & X_{3} \\ \hline
E & (1,0,2) & (0,1) & (0) \\
F & (0) & (0) & (0) \\
G & (0) & (0,1) & (1,0,0,-1)
\end{array}
$$
With respect to these bases, define $E$ by
\begin{align*}
A_{1}^{E}=\begin{pmatrix}0&1&0\\-1&0&0\end{pmatrix},\quad
A_{2}^{E}=\begin{pmatrix}0&1&0\\-2&0&0\end{pmatrix},\quad
A_{3}^{E}=\begin{pmatrix}0&0&0\\0&0&0\end{pmatrix},\quad 
A_{4}^{E}=\begin{pmatrix}-2&0&0\\0&0&2\end{pmatrix},\\
B_{1}^{E}=\begin{pmatrix}-1&0\end{pmatrix},\quad 
B_{2}^{E}=\begin{pmatrix}-1&0\end{pmatrix},\quad 
B_{3}^{E}=\begin{pmatrix}0&-1\end{pmatrix}, \quad 
B_{4}^{E}=\begin{pmatrix}0&-2\end{pmatrix},\quad 
B_{5}^{E}=\begin{pmatrix}0&0\end{pmatrix}.
\end{align*}
Define $F$ by
\begin{align*}
&A_{1}^{F}=(-2),\qquad 
A_{2}^{F}=(-1),\qquad 
A_{3}^{F}=A_{4}^{F}=(0), \\
&B_{1}^{F}=(-1),\qquad 
B_{2}^{F}=(-2),\qquad 
B_{3}^{F}=B_{4}^{F}=B_{5}^{F}=(0).
\end{align*}
Finally, define $G$ by
\begin{align*}
& A_{1}^{G}=\begin{pmatrix}-2\\0\end{pmatrix},\quad
A_{2}^{G}=\begin{pmatrix}0\\0\end{pmatrix},\quad
A_{3}^{G}=\begin{pmatrix}0\\-1\end{pmatrix},\quad
A_{4}^{G}=\begin{pmatrix}0\\0\end{pmatrix},\\
& B_{1}^{G}=\begin{pmatrix}0&2\\1&0\\0&0\\0&0\end{pmatrix},\quad
B_{2}^{G}=\begin{pmatrix}0&1\\1&0\\1&0\\0&0\end{pmatrix},\quad
B_{3}^{G}=\begin{pmatrix}0&0\\0&1\\0&1\\1&0\end{pmatrix},
\quad B_{4}^{G}=\begin{pmatrix}0&0\\0&-1\\0&-2\\-1&0\end{pmatrix},\quad
B_{5}^{G}=\begin{pmatrix}0&0\\0&0\\0&0\\0&-2\end{pmatrix}.
\end{align*}
All these matrices are homogeneous for the internal grading.  Moreover,
they satisfy the 7 defining relations of $A$:
$$
\sum_{j,i}(\rho_{\ell})_{ji}B_{j}^{X}A_{i}^{X}=0
\qquad
(X=E,F,G,\;1\leq\ell\leq7).
$$
So $E,F,G$ are well-defined internally graded $A$-modules.

\subsection{Graded Ext groups}\label{section-exceptional}
In this subsection we verify that $(E,F,G)$ is an exceptional triple and record the internally graded
$\Ext$ groups that will be used later.

To obtain $A$-projective resolutions of $E, F, G$, we first obtain a projective $A^{e}$-resolution of $A$.
Let $e_{i}$ be the idempotent of vertex $i$, and set
$ S=ke_{1}\oplus ke_{2}\oplus ke_{3}$.  Recall that $U,V,W$
are defined in (\ref{equation-UVW}). We have $J=U\oplus V\oplus W$. Thus
$$
U=e_{2}Je_{1},\quad
V=e_{3}Je_{2},\quad
W=e_{3}Je_{1}; \qquad
J^{2}=W, \quad
J^{3}=0.
$$
It follows that
$
J\tensor_{S} J\cong V\tensor_{S} U.
$
Multiplication in $A$ induces a surjective map
$$
\mu:V\tensor_{S} U\longrightarrow W,
\qquad
b_{j}\tensor a_{i}\longmapsto b_{j}a_{i}.
$$
By definition,
$\ker(\mu)=R$, and we have the exact sequence
\begin{equation}\label{equation-relation-sequence}
0\longrightarrow R
\longrightarrow V\tensor_{S} U
\xrightarrow{\mu} W
\longrightarrow0.
\end{equation}
Since
$
J^{\otimes_{S} 3}=0
$,
the relative bar resolution of $A$ over $S$ terminates after the
second term:
\begin{equation}\label{equation-relative-bar}
0\longrightarrow
A\tensor_{S} J\tensor_{S} J\tensor_{S} A
\longrightarrow
A\tensor_{S} J\tensor_{S} A
\longrightarrow
A\tensor_{S} A
\longrightarrow
A
\longrightarrow0.
\end{equation}
Now we can simplify this resolution. 
Since $S$ is semisimple,
(\ref{equation-relation-sequence}) admits a graded $S$-bimodule splitting.
Choose a graded $S$-bimodule $R'$ such that
$
V\tensor_{S} U=R\oplus R'
$
and such that the restriction of $\mu$ induces an isomorphism
$
\mu|_{R'}:R'\xrightarrow{\sim}W.
$
Using
$$
J\tensor_{S} J\cong R\oplus R'
\qquad\text{and}\qquad
J=(U\oplus V)\oplus W,
$$
the first two terms of (\ref{equation-relative-bar}) decompose as
\begin{equation}
\begin{array}{rl}
A\tensor_{S} J\tensor_{S} J\tensor_{S} A
\cong &
(A\tensor_{S} R\tensor_{S} A)
\oplus
(A\tensor_{S} R'\tensor_{S} A), \\
A\tensor_{S} J\tensor_{S} A
\cong &
\bigl(A\tensor_{S}(U\oplus V)\tensor_{S} A\bigr)
\oplus
(A\tensor_{S} W\tensor_{S} A).
\end{array}
\end{equation}
With respect to these decompositions, the bar differential has a component
$$
A\tensor_{S} R'\tensor_{S} A
\longrightarrow
A\tensor_{S} W\tensor_{S} A
$$
induced by the isomorphism $\mu|_{R'}:R'\xrightarrow{\sim}W$.  Hence this component is an isomorphism, and cancelling them in (\ref{equation-relative-bar})
gives the resolution
\begin{equation}\label{equation-minimal-bimodule-resolution}
0\longrightarrow
A\tensor_{S} R\tensor_{S} A
\longrightarrow
A\tensor_{S}(U\oplus V)\tensor_{S} A
\longrightarrow
A\tensor_{S} A
\longrightarrow
A
\longrightarrow0.
\end{equation}

\begin{lemma}\label{lemma-euler-form}
Let $X,Y$ be finite-dimensional left $A$-modules with
$$
\overrightarrow{\dim}(X)=(x_{1},x_{2},x_{3}),
\qquad
\overrightarrow{\dim}(Y)=(y_{1},y_{2},y_{3}).
$$
The Euler pairing is given by
\begin{equation}\label{equation-euler-form}
\begin{array}{rl}
\chi(X,Y)
:=&
\sum_{i\geq0}(-1)^{i}\dim_{k}\Ext_{A}^{i}(X,Y)\\
=&
x_{1}y_{1}+x_{2}y_{2}+x_{3}y_{3}
-4x_{1}y_{2}-5x_{2}y_{3}+7x_{1}y_{3}.
\end{array}
\end{equation}
\end{lemma}

\begin{proof}
Applying $-\otimes_{A}X$ to 
(\ref{equation-minimal-bimodule-resolution}) gives
\begin{equation}\label{equation-resolution-X}
0\longrightarrow
A\tensor_{S} R\tensor_{S} X
\longrightarrow
A\tensor_{S}(U\oplus V)\tensor_{S} X
\longrightarrow
A\tensor_{S} X
\longrightarrow
X
\longrightarrow0.
\end{equation}
Since $S$ is semisimple, every finite-dimensional $S$-module is projective,
and each term in (\ref{equation-resolution-X}) is therefore a finite direct sum of
the projective left $A$-modules $Ae_{1},Ae_{2},Ae_{3}$. So
(\ref{equation-resolution-X}) is a projective resolution of $X$. Applying $\Hom_{A}(-,Y)$ to (\ref{equation-resolution-X}) gives 
\begin{equation}\label{equation-Ext-complex}
\cK^{0}(X,Y)\xrightarrow{d^{0}}
\cK^{1}(X,Y)\xrightarrow{d^{1}}
\cK^{2}(X,Y).
\end{equation}
The cohomology of (\ref{equation-Ext-complex}) is $\Ext_{A}^{\bullet}(X,Y)$. Since
$A\tensor_{S} X
\cong
\bigoplus_{v=1}^{3} Ae_{v}\tensor_{k} X_{v}$,
we have
$$
\cK^{0}(X,Y)
\cong
\bigoplus_{v=1}^{3}\Hom_{k}(X_{v},Y_{v}).
$$
Next, using
$
U=e_{2}Ue_{1},$ $
V=e_{3}Ve_{2},
$
together with $\dim_{k}U=4$ and $\dim_{k}V=5$, we obtain
$$
\cK^{1}(X,Y)
\cong
\Hom_{k}(X_{1},Y_{2})^{\oplus4}
\oplus
\Hom_{k}(X_{2},Y_{3})^{\oplus5}.
$$
Finally, since
$
R=e_{3}Re_{1}$
 and
$\dim_{k}R=7,
$
we have
$$
\cK^{2}(X,Y)
\cong
\Hom_{k}(X_{1},Y_{3})^{\oplus7}.
$$
Hence we have the dimensions
\begin{align*}
\dim_{k} \cK^{0}(X,Y)
&=x_{1}y_{1}+x_{2}y_{2}+x_{3}y_{3},\\
\dim_{k} \cK^{1}(X,Y)
&=4x_{1}y_{2}+5x_{2}y_{3},\\
\dim_{k} \cK^{2}(X,Y)
&=7x_{1}y_{3}.
\end{align*}
Therefore, we obtain
$$\chi(X,Y) =x_{1}y_{1}+x_{2}y_{2}+x_{3}y_{3}
-4x_{1}y_{2}-5x_{2}y_{3}+7x_{1}y_{3}.$$
\end{proof}

\begin{lemma}\label{lemma-exceptional-triple}
The triple $(E,F,G)$ is an exceptional collection in $\Perf(A)$.
\end{lemma}

\begin{proof}
We just need to make the differentials in (\ref{equation-Ext-complex}) explicit and then compute the cohomology.
For $\phi=(\phi_{1},\phi_{2},\phi_{3})\in \cK^{0}(X,Y)$, the first differential in
(\ref{equation-Ext-complex}) is
$$(d^{0}\phi)_{a_{i}}=A_{i}^{Y}\phi_{1}-\phi_{2}A_{i}^{X}, \qquad
(d^{0}\phi)_{b_{j}}=B_{j}^{Y}\phi_{2}-\phi_{3}B_{j}^{X}. $$
Recall the relations $\rho_{1},\ldots,\rho_{7}$ from
Section~\ref{section-algebra}.
Writing
$\rho_{\ell}=\sum_{j,i}(\rho_{\ell})_{ji}b_{j}a_{i}$,
the second differential is
$$
(d^{1}\xi)_{\ell}
=
\sum_{j,i}(\rho_{\ell})_{ji}
\bigl(B_{j}^{Y}\xi_{a_{i}}+\xi_{b_{j}}A_{i}^{X}\bigr).
$$
Using the explicit matrices defining $E,F,G$, direct rank computations give
\begin{align*}
&\Ext_{A}^{\bullet}(E,E)
\cong
\Ext_{A}^{\bullet}(F,F)
\cong
\Ext_{A}^{\bullet}(G,G)
\cong k
\\
&\Ext_{A}^{\bullet}(F,E)
=
\Ext_{A}^{\bullet}(G,E)
=
\Ext_{A}^{\bullet}(G,F)
=
0.
\end{align*}
\end{proof}

\begin{remark}\label{remark-upper-triangular}
For $(X_{1}, X_{2}, X_{3})=(E, F, G)$,
Lemma~\ref{lemma-euler-form} computes the Euler matrix
$$
\bigl(\chi(X_{i},X_{j})\bigr)_{i,j=1}^{3}
=
\begin{pmatrix}
1&5&31\\
0&1&7\\
0&0&1
\end{pmatrix}.
$$
\end{remark}

For later use, we record the dimensions and ranks of the Ext complexes for
all pairs among $E,F,G$:
\begin{center}
\begin{equation}\label{table-ext}
\small
\begin{tabular}{c|c|c|c}
$(X,Y)$ & $(\dim \cK^{0},\dim \cK^{1},\dim \cK^{2})$
& $(\rk d^{0},\rk d^{1})$
& $(\dim\Ext^{0},\dim\Ext^{1},\dim\Ext^{2})$\\ \hline
$(E,E)$&(14,34,21)&(13,21)&(1,0,0)\\
$(F,F)$&(3,9,7)&(2,7)&(1,0,0)\\
$(G,G)$&(21,48,28)&(20,28)&(1,0,0)\\
$(F,E)$&(6,13,7)&(6,7)&(0,0,0)\\
$(G,E)$&(11,18,7)&(11,7)&(0,0,0)\\
$(G,F)$&(7,14,7)&(7,7)&(0,0,0)\\
$(E,F)$&(6,22,21)&(6,16)&(0,0,5)\\
$(E,G)$&(11,64,84)&(11,53)&(0,0,31)\\
$(F,G)$&(7,28,28)&(7,21)&(0,0,7)
\end{tabular}
\end{equation}
\end{center}

The last 3 rows of (\ref{table-ext}) show that the forward Ext groups are
concentrated in degree $2$.  For later use, we will make use of their internal gradings.
Set
$$
f_{XY}(q)=\ch\Ext_{A}^{2}(X,Y),
\qquad
(X,Y)=(E,F),(E,G),(F,G).
$$
A direct computation gives the following enhancement of dimensions into internal characters:
\begin{equation}\label{equation-characters-f}
\begin{aligned}
f_{EF}(q)&=2+2q^{-1}+q^{-2},\\
f_{EG}(q)&=2q+7+9q^{-1}+8q^{-2}+4q^{-3}+q^{-4},\\
f_{FG}(q)&=q+3+2q^{-1}+q^{-2}.
\end{aligned}
\end{equation}

\section{Regrading}\label{section-regrading}

In this section, we introduce a regrading construction and apply it to the algebra $A$ (\ref{equation-def-A}) and the modules $E,F,G$ constructed
in subsection \ref{section-modules}. This produces the family of DG algebras and
residual categories appearing in Theorem~\ref{theorem-main}. We start with the general construction.

Let
$
V=\bigoplus_{p,w}V^{p,w}
$
be a bigraded complex, where $p$ denotes the cohomological degree and $w$ the
internal degree.  For $m\in\ZZ$, define the regrading $T_{m}V$ by
\begin{equation}\label{equation-regrading}
(T_{m}V)^{d}
=
\bigoplus_{p+2mw=d}V^{p,w}.
\end{equation}
Thus the bidegree $(p,w)$ is sent to the cohomological degree
$
p+2mw.
$
First, we note the compatibility conditions in the following lemma.
\begin{lemma}\label{lemma-regrading}
Let $V=\bigoplus_{p,w}V^{p,w}$ be a finite bigraded complex whose differential
has bidegree $(1,0)$.  Then $T_{m}$ is exact and preserves parity.  Moreover,
there are natural isomorphisms
$$
T_{m}(V\tensor Z)\cong T_{m}V\tensor T_{m}Z,
\qquad
T_{m}\Hom(V,Z)\cong\Hom(T_{m}V,T_{m}Z),
$$
together with analogous identifications for duals, shifts, and cones.
In particular, $T_{m}$ preserves quasi-isomorphisms.
\end{lemma}

\begin{proof}
 Since the correction $2mw$ is even, all Koszul
signs are preserved.
\end{proof}

\subsection{The family $A_{m}$ and $\cP_{m}$}

In this subsection, we apply the regrading (\ref{equation-regrading}) to $A$ and $(E,F,G)$, producing $A_{m}$ and $(E_{m},F_{m},G_{m})$ for every $m\in\ZZ$.

The algebra $A$ and modules $E,F,G$ have cohomological degree
0, together with the internal grading introduced in
Section~\ref{section-fixed-data}. So every homogeneous element initially has
bidegree $(0,w)$, and under $T_{m}$ it is regraded in cohomological degree $2mw$.
For $m\in\ZZ$, set
$$
A_{m}=T_{m}A,\qquad
E_{m}=T_{m}E,\qquad
F_{m}=T_{m}F,\qquad
G_{m}=T_{m}G.
$$
The arrow degrees in $A_{m}$ are
$$
\begin{array}{c|rrrrrrrrr}
&a_{1}&a_{2}&a_{3}&a_{4}&b_{1}&b_{2}&b_{3}&b_{4}&b_{5}\\ \hline
\deg_{A_{m}}&0&0&2m&-2m&0&0&-2m&-2m&-4m.
\end{array}
$$
The regraded representations of $(E, F, G)$ have underlying complexes
$$
\begin{array}{c|ccc}
 &X_{1}&X_{2}&X_{3}\\ \hline
E_{m}&k[-2m]\oplus k\oplus k[-4m]&k\oplus k[-2m]&k\\
F_{m}&k&k&k\\
G_{m}&k&k\oplus k[-2m]&k[-2m]\oplus k^{\oplus2}\oplus k[2m].
\end{array}
$$

\begin{lemma}\label{lemma-regraded-exceptional}
For every $m\in\ZZ$, the triple $(E_{m},F_{m},G_{m})$ is exceptional in
$\Perf(A_{m})$.
\end{lemma}

\begin{proof}
By Lemma~\ref{lemma-exceptional-triple}, $(E,F,G)$ is exceptional in
$\Perf(A)$. By Lemma~\ref{lemma-regrading}, the corresponding Ext complexes
are carried to those of $(E_{m},F_{m},G_{m})$ under regrading. The self-Ext class is generated by the identity map, which has bidegree $(0,0)$. So it remains in cohomological degree $0$ for any regrading. The 3 backward $\Ext$ groups remain zero. (The forward $\Ext$ groups need not remain in degree $2$.)
\end{proof}

 \subsection{The residual categories $\cP_{m}$}\label{section-common-model}

In this subsection, we define the residual categories $\cP_{m}$ and show that they are all obtained by regrading a single finite-dimensional bigraded DG algebra.  This gives a common model for the family that will be used in the Hochschild cohomology computation.

Regrading the standard full exceptional collection
$
\Perf(A)=\langle P_{3},P_{2},P_{1} \rangle
$ gives a full exceptional collection 
$$
\Perf(A_{m})=\langle P_{3}^{(m)},P_{2}^{(m)},P_{1}^{(m)} \rangle,
\qquad
P_{i}^{(m)}=A_{m}e_{i}.
$$  
By Lemma~\ref{lemma-regraded-exceptional},
$\langle E_{m},F_{m},G_{m} \rangle\subset \Perf(A_{m})$ is an exceptional collection. Set
$$
\cP_{m}=\langle E_{m},F_{m},G_{m}\rangle^{\perp}.
$$
We next show that all $\cP_{m}$ are obtained by regrading
a single bigraded DG algebra.

\begin{proposition}\label{proposition-common-model}
There exists a finite-dimensional bigraded DG algebra $B$, whose differential
has bidegree $(1,0)$, such that for every $m\in\ZZ$,
$$
\cP_{m}\simeq_{\mathrm{Morita}}\Perf(T_{m}B).
$$
\end{proposition}

\begin{proof}
We first take $B$ as the standard DG algebra model of $\cP_{0}$. Specifically, we let
$
\pi_{0}:\Perf(A)\longrightarrow\cP_{0}
$
be the projection associated with 
$\langle\cP_{0},E,F,G\rangle.$
Set
$T=\bigoplus_{i=1}^{3}\pi_{0}(P_{i})$.
Since $P_{1},P_{2},P_{3}$ generate $\Perf(A)$, their projections generate
$\cP_{0}$.  Hence $T$ is a classical generator of $\cP_{0}$. Choose a bounded finite internally graded projective representative of $T$
and set
$
B=\End_{A}^{\bullet}(T)^{\mathrm{op}}.
$
Then $B$ is a finite-dimensional bigraded DG algebra whose differential has
bidegree $(1,0)$, and
$
\cP_{0}\simeq_{\mathrm{Morita}}\Perf(B).
$

The projection $\pi_{0}$ is obtained by successive mutations through
$E,F,G$.  These mutations are constructed from $\RHom$, tensor products,
evaluation maps, and cones.  By Lemma~\ref{lemma-regrading}, the regrading
functor $T_{m}$ is compatible with these operations.  Therefore $T_{m}(T)$ is
the projection of
$
P_{1}^{(m)}\oplus P_{2}^{(m)}\oplus P_{3}^{(m)}
$
to $\cP_{m}$.  Since $P_{1}^{(m)},P_{2}^{(m)},P_{3}^{(m)}$ generate
$\Perf(A_{m})$, the object $T_{m}(T)$ generates $\cP_{m}$. Moreover,
$$
\RHom_{A_{m}}\bigl(T_{m}(T),T_{m}(T)\bigr)^{\mathrm{op}}
\cong
T_{m}B.
$$
It follows that
$
\cP_{m}\simeq_{\mathrm{Morita}}\Perf(T_{m}B)
$.
\end{proof}

\section{Hochschild cohomology}\label{section-height}
The goal of this section is to compute the Hochschild cohomology $\HoH^{*}(\cP_{m})$ for the regraded residual categories.  We first refine Hochschild cohomology by the internal grading, then compute the normal Hochschild terms needed for the residual category, and finally apply the regrading formula.

\subsection{Internal grading refinement}

Since we will deal with infinitely many regraded categories $\cP_{m}$, the goal of this subsection is to refine Hochschild cohomology by the internal grading, and prove the regrading formula that recovers ordinary Hochschild cohomology after applying $T_{m}$.

\begin{definition}
Let $C$ be a finite-dimensional smooth bigraded DG algebra. For cohomological degree
$p$ and internal weight $w$, define
$$
\HoH^{p,w}(C)
:=
H^{p}\!\left(\RHom_{C^{e}}(C,C)\right)_{w}.
$$
For the residual category $\cP_{0}$, we use the
model $B$ in Proposition~\ref{proposition-common-model} and set
$$
\HoH^{p,w}(\cP_{0}):=\HoH^{p,w}(B).
$$
\end{definition}
\begin{lemma}
\label{lemma-graded-perfectness}
Let
$
R=\bigoplus_{w\in\ZZ}R^{\bullet,w}
$
be a DG algebra equipped with an internal \(\ZZ\)-grading, such that
the differential has internal degree \(0\) and multiplication is homogeneous
for the internal grading. Let
$$
\operatorname{For}\colon D^{\mathrm{gr}}(R)\longrightarrow D(R)
$$
be the functor forgetting the internal grading.

If \(M\in D^{\mathrm{gr}}(R)\) and \(\operatorname{For}(M)\) is perfect in
\(D(R)\), then \(M\) is perfect in \(D^{\mathrm{gr}}(R)\).
\end{lemma}

\begin{proof}
Since \(\operatorname{For}(M)\) is perfect in \(D(R)\), it is compact;
see for example \cite[\S5]{KellerDerivingDG}. We first show that \(M\)
is compact in \(D^{\mathrm{gr}}(R)\).

Let \(\{N_{i}\}_{i\in I}\) be a family of internally graded DG \(R\)-modules.
We need to show that the natural map
\begin{equation}
\label{equation-graded-compactness-map}
\bigoplus_{i\in I}
\Hom_{D^{\mathrm{gr}}(R)}(M,N_{i})
\longrightarrow
\Hom_{D^{\mathrm{gr}}(R)}
\left(M,\bigoplus_{i\in I}N_{i}\right)
\end{equation}
is an isomorphism. Injectivity is immediate by composing with the projections
$$
\bigoplus_{i\in I}N_{i}\longrightarrow N_{i}.
$$

For surjectivity of (\ref{equation-graded-compactness-map}), choose a homogeneous
semi-free resolution
$
P\longrightarrow M
$
in the category of internally graded DG \(R\)-modules. Then \(P\) is
h-projective in the graded category. After forgetting the internal grading,
\(\operatorname{For}(P)\) is still semi-free, hence h-projective, and
$
\operatorname{For}(P)\longrightarrow \operatorname{For}(M)
$
is an h-projective resolution. Let
$$
\alpha\colon M\longrightarrow\bigoplus_{i\in I}N_{i}
$$
be a morphism in \(D^{\mathrm{gr}}(R)\). Since \(P\) is h-projective,
\(\alpha\) is represented by an \(R\)-linear chain map
$$
a\colon P\longrightarrow\bigoplus_{i\in I}N_{i}
$$
of internal degree \(0\). Since \(\operatorname{For}(M)\) is compact in \(D(R)\), the morphism
\(\operatorname{For}(\alpha)\) factors through a finite subcoproduct.
Thus there exist a \emph{finite} subset \(J\subset I\) and a morphism
$$
\beta\colon \operatorname{For}(M)
\longrightarrow
\bigoplus_{j\in J}\operatorname{For}(N_{j})
$$
in \(D(R)\) such that the composition
$$
\operatorname{For}(M)
\xrightarrow{\beta}
\bigoplus_{j\in J}\operatorname{For}(N_{j})
\xrightarrow{\iota_{J}}
\bigoplus_{i\in I}\operatorname{For}(N_{i})
$$
is equal to \(\operatorname{For}(\alpha)\). Since \(\operatorname{For}(P)\) is h-projective, we may represent
\(\beta\) by an \(R\)-linear chain map
$$
b\colon P\longrightarrow\bigoplus_{j\in J}N_{j}.
$$
Viewing \(a\) and \(\iota_{J}b\) as ungraded chain maps, they represent the
same morphism in \(D(R)\). Hence they are homotopic: there exists an
\(R\)-linear homotopy
\begin{equation}
\label{equation-graded-homotopy}
h\colon
P\longrightarrow
\bigoplus_{i\in I}N_{i}[-1],\qquad a-\iota_{J}b=dh+hd.
\end{equation}

We now take the internal-degree-zero components of \(b\) and \(h\).
For a homogeneous element \(x\in P^{\bullet,w}\), define
$$
b_{0}(x):=\pr_{w}\bigl(b(x)\bigr),
\qquad
h_{0}(x):=\pr_{w}\bigl(h(x)\bigr),
$$
where \(\pr_{w}\) denotes projection to internal weight
\(w\). Since the differential has internal degree \(0\) and the
\(R\)-action is homogeneous, \(b_{0}\) is an \(R\)-linear chain map of
internal degree \(0\), and \(h_{0}\) is an \(R\)-linear homotopy of internal
degree \(0\). Taking the internal-degree-zero component of
(\ref{equation-graded-homotopy}) gives
$$
a-\iota_{J}b_{0}=dh_{0}+h_{0}d.
$$
Hence \(\alpha\) factors in \(D^{\mathrm{gr}}(R)\) through 
$
\bigoplus_{j\in J}N_{j}.
$
Therefore (\ref{equation-graded-compactness-map}) is surjective, and \(M\) is
compact in \(D^{\mathrm{gr}}(R)\).
\end{proof}

\begin{proposition}\label{proposition-HH-regrading}
Let $C$ be a finite-dimensional smooth bigraded DG algebra whose
differential has bidegree $(1,0)$.  Then, for every $m\in\ZZ$, there are
natural isomorphisms
$$
\HoH^{d}(T_{m}C)\cong
\bigoplus_{p+2mw=d}\HoH^{p,w}(C),
$$
where the right-hand side is computed from any finite homogeneous
$C^{e}$-projective resolution of $C$.
\end{proposition}

\begin{proof}
By Lemma \ref{lemma-graded-perfectness}, we may choose an \emph{internally graded} finite
$C^{e}$-projective resolution $K\to C$. The (bigraded) complex
$\Hom_{C^{e}}(K,C)$
computes $\HoH^{\bullet}(C)$, with the second grading recording internal
weight. Lemma~\ref{lemma-regrading}
gives an isomorphism of complexes
$$
\Hom_{(T_{m}C)^{e}}(T_{m}K,T_{m}C)
\cong T_{m}\Hom_{C^{e}}(K,C).
$$
Here $T_{m}(C^{e})$ is identified with $(T_{m}C)^{e}$.
Taking cohomology and then the degree-$d$ summand gives the formula.
\end{proof}

\subsection{Projective resolutions and inverse-Serre Ext groups}

The main purpose of this subsection is to compute the $\Ext$ groups needed in the normal Hochschild spectral sequence, together with their internal graded characters.

Set $\cB=\langle E,F,G\rangle\subset\Perf(A)$. Since $\cB$ is generated by an exceptional collection, it is admissible, and $\cB^{\perp}=\cP_{0}$.
Recall that Kuznetsov's normal Hochschild triangle~\cite{KuznetsovHeight} gives
\begin{equation}\label{equation-NHH-triangle}
\NHH^{\bullet}(\cB,A)\longrightarrow
\HoH^{\bullet}(A)\longrightarrow
\HoH^{\bullet}(\cB^{\perp}).
\end{equation}
The reduced bar spectral sequence for $\NHH$ has
\begin{equation}\label{equation-NHH-ss}
E_{1}^{-\ell,q}
=
\bigoplus_{
i_{0}<\cdots<i_{\ell}
}
\bigoplus_{k_{0}+\cdots+k_{\ell}=q}
\Ext^{k_{0}}(X_{i_{0}},X_{i_{1}})\tensor\cdots\tensor
\Ext^{k_{\ell-1}}(X_{i_{\ell-1}},X_{i_{\ell}})\tensor
\Ext^{k_{\ell}}(X_{i_{\ell}},S_{A}^{-1}X_{i_{0}}).
\end{equation}
Since we want to obtain the Hochschild cohomology refined by internal gradings, the $\Ext$ dimensions are upgraded to internal characters.

Recall from Section~\ref{section-algebra} that $U=e_{2}Qe_{1}$ is the arrow space of
$1\to 2$, and $V=e_{3}Qe_{2}$ is the arrow space of $2\to3$. Let $R\subset V\tensor U$ be the subspace spanned by the relations on $Q$, and $W=(V\tensor U)/R$ is the space of surviving paths of $1\to3$.
For any representation $X\in\{E,F,G\}$, define the following evaluation maps
$$
\alpha_{X}:U\tensor X_{1}\longrightarrow X_{2},
\quad
a_{i}\tensor x\longmapsto A_{i}^{X}x, \qquad \gamma_{X}:W\tensor X_{1}\longrightarrow X_{3}, \quad b_{j}a_{i}\otimes x\mapsto B^{X}_{j}A_{i}^{X} x.
$$
Let $\pi:V\tensor U\twoheadrightarrow W$ be the quotient map. Define
$$
L_{X}=\ker\alpha_{X},
\qquad
\beta_{X}=(\pi\tensor1_{X_{1}})|_{V\tensor L_{X}},
\qquad
N_{X}=\ker\beta_{X}.
$$
Thus $L_{X}$ records relations among the generators coming from arrows
$1\to2$, and $N_{X}$ records the next syzygies. By concrete linear algebra, we record the following table of data:
$$
\begin{array}{c|ccccc}
X&\rk\alpha_{X}&\rk\gamma_{X}&\rk\beta_{X}
&\dim L_{X}&\dim N_{X}\\ \hline
E&2&1&38&10&12\\
F&1&1&12&3&3\\
G&2&4&9&2&1.
\end{array}
$$
Consequently, for each $X\in\{E,F,G\}$ there is a projective resolution
$$
0\longrightarrow P_{3}\tensor N_{X}
\longrightarrow P_{2}\tensor L_{X}
\longrightarrow P_{1}\tensor X_{1}
\longrightarrow X\longrightarrow0,
\qquad P_{i}=Ae_{i}.
$$
The corresponding projective and
injective resolutions have the numerical data:
$$
\begin{array}{c|ccc|ccc}
X&P_{1}^{0}&P_{2}^{-1}&P_{3}^{-2}&I_{3}^{0}&I_{2}^{1}&I_{1}^{2}\\ \hline
E&3&10&12&1&3&2\\
F&1&3&3&1&4&4\\
G&1&2&1&4&18&21.
\end{array}
$$

Applying these resolutions to
$\RHom_{A}(X,S_{A}^{-1}Y)$ produces finite complexes concentrated in
cohomological degrees $2,3,4$. Concrete linear algebra gives the following table of data:
\begin{equation}\label{table-inverse-serre}
\begin{array}{c|c|c|c|c}
X&Y&(\dim C^{2},\dim C^{3},\dim C^{4})
&(\rk d^{2},\rk d^{3})&\dim\Ext^{4}(X,S_{A}^{-1}Y)\\ \hline
E&E&(48,260,312)&(48,212)&100\\
F&F&(19,108,156)&(19,89)&67\\
G&G&(61,258,273)&(61,197)&76\\
F&E&(14,69,78)&(14,55)&23\\
G&E&(9,31,26)&(9,22)&4\\
G&F&(13,52,52)&(13,39)&13.
\end{array}
\end{equation}
So every inverse-Serre complex needed below has cohomology concentrated in degree 4.
The internal grading refines these dimensions. For the 6 pairs above,
set
$$
v_{XY}(q)=\ch\Ext^{4}_{A}(X,S_{A}^{-1}Y).
$$
Then we have the following character polynomials:
\begin{equation}\label{equation-characters-v}
\begin{array}{rl}
v_{EE}&=q^{-4}+6q^{-3}+15q^{-2}+23q^{-1}+25+19q+9q^{2}+2q^{3},\\
v_{FF}&=q^{-4}+5q^{-3}+11q^{-2}+16q^{-1}+16+12q+5q^{2}+q^{3},\\
v_{GG}&=q^{-4}+5q^{-3}+12q^{-2}+18q^{-1}+19+14q+6q^{2}+q^{3},\\
v_{FE}&=q^{-2}+4q^{-1}+6+7q+4q^{2}+q^{3},\\
v_{GE}&=1+2q+q^{2},\\
v_{GF}&=q^{-2}+3q^{-1}+4+4q+q^{2}.
\end{array}
\end{equation}
Forgetting the internal grading is equivalent to setting $q=1$. It can be checked that this is consistent with the dimension table: for
example, $v_{EE}(1)=100$. Note that the
characters of tensor products multiply.

\subsection{Normal Hochschild cohomology and height}
Gathering the ingredients computed above, we now use (\ref{equation-NHH-ss}) to prove Theorem \ref{theorem-main}.

In the spectral sequence (\ref{equation-NHH-ss}), since there are only 3
objects, we have $\ell=0,1,2$. By (\ref{table-ext}) and (\ref{table-inverse-serre}), the only nonzero positions are the following:
$$
\begin{array}{c|c|c}
\ell&\text{type of tensor product}&(-\ell,q)\\ \hline
0&\Ext^{4}(X_{i},S_{A}^{-1}X_{i})&(0,4)\\
1&\Ext^{2}(X_{i},X_{j})\tensor\Ext^{4}(X_{j},S_{A}^{-1}X_{i})&(-1,6)\\
2&\Ext^{2}(E,F)\tensor\Ext^{2}(F,G)\tensor\Ext^{4}(G,S_{A}^{-1}E)&(-2,8).
\end{array}
$$
There are no possible nonzero differentials between these positions, so
the spectral sequence degenerates at $E_{1}$.  Thus
$$
\NHH^{t}(\cB,A)=0\qquad(t\notin\{4,5,6\}),
\qquad
\dim_{k}\NHH^{4}(\cB,A)=100+67+76=243,
$$
and the height is 4.

Keeping track of the internal grading in the three surviving
$E_{1}$-terms gives
\begin{equation}
    \begin{array}{rl}
\ch \NHH^{4,*}(\cB,A)
=&
v_{EE}+v_{FF}+v_{GG},\\
\ch \NHH^{5,*}(\cB,A)
=&
f_{EF}v_{FE}+f_{EG}v_{GE}+f_{FG}v_{GF},
\\
\ch \NHH^{6,*}(\cB,A)
=&
f_{EF}f_{FG}v_{GE}.
\end{array}
\end{equation}

For completeness, the Hochschild cochain complex of the ambient algebra
$A$ is quasi-isomorphic to the following complex
$$
C^{0}=k^{3}\xrightarrow{\delta^{0}}
C^{1}=\End(U)\oplus\End(V)\xrightarrow{\delta^{1}}
C^{2}=\Hom(R,W),
$$
where
$$
\delta^{1}(P,Q)(r)=\pi\bigl((Q\tensor1_{U}+1_{V}\tensor P)r\bigr).
$$
By concrete linear algebra, we have 
\begin{equation}\label{equation-HH-012}
\begin{array}{l}
\HoH^{0,0}(A)\cong k, \quad
\HoH^{2,-4}(A)\cong k, \quad
\HoH^{2,3}(A)\cong k,
\\
\HoH^{0,w}(A)=\HoH^{1,w}(A)=0, ~\forall w\neq0.
\end{array}
\end{equation}
The full block table is recorded in a separate computational companion file.

Since the Hochschild complex of $A$ is concentrated in degrees
$0,1,2$, the long exact sequence induced by (\ref{equation-NHH-triangle}) gives
\begin{equation}\label{equation-HH-P0}
\HoH^{p,w}(\cP_{0})\cong
\begin{cases}
\HoH^{p,w}(A),&p=0,1,2,\\
\NHH^{p+1,w}(\cB,A),&p=3,4,5,\\
0,&\text{otherwise}.
\end{cases}
\end{equation}
Therefore we have
\begin{equation}
\begin{array}{rl}
\ch \HoH^{3,*}(\cP_{0})=&v_{EE}+v_{FF}+v_{GG},
\\
\ch \HoH^{4,*}(\cP_{0})=&f_{EF}v_{FE}+f_{EG}v_{GE}+f_{FG}v_{GF},
\\
\ch \HoH^{5,*}(\cP_{0})=&f_{EF}f_{FG}v_{GE}.
\end{array}
\end{equation}

\begin{lemma}\label{lemma-HH-endpoints}
For
$\cP_{0}=\langle E,F,G\rangle^{\perp}\subset\Perf(A)$
we have
$$
\HoH^{p,w}(\cP_{0})=0
\qquad\text{unless}\qquad
0\leq p\leq5,\quad -4\leq w\leq3.
$$
Moreover, we have
$$
\HoH^{0,0}(\cP_{0})\cong k,
\qquad
\HoH^{2,-4}(\cP_{0})\cong k,
\qquad
\HoH^{2,3}(\cP_{0})\cong k,
$$
$$
\HoH^{5,-4}(\cP_{0})\cong k,
\qquad
\HoH^{5,3}(\cP_{0})\cong k^{2},
$$
and if $w<0$, then $\HoH^{p,w}(\cP_{0})=0$ for $p<2$.
\end{lemma}

\begin{proof}
When $p=3,4,5$, the claim follows from (\ref{equation-HH-P0}), (\ref{equation-characters-f}) and (\ref{equation-characters-v}). The other claims follow from (\ref{equation-HH-012}). 
\end{proof}

\begin{proof}[Proof of Theorem~\ref{theorem-main}]
By Lemma \ref{lemma-regraded-exceptional}, $\langle E_{m}, F_{m}, G_{m} \rangle $ is an exceptional collection of length 3. Since $A_{m}$ is smooth and proper, the exceptional collection is
admissible, so the residual category $\cP_{m}$ is smooth and proper. Recall that $(P_{3}^{(m)},P_{2}^{(m)},P_{1}^{(m)})$ is a full exceptional collection, so $K_{0}(A_{m})\cong\ZZ^{3}$ and $\HoH_*(A_{m})\cong k^{3}$. The Euler matrices of both the standard collection and $\langle E_{m},F_{m},G_{m}\rangle$ are upper triangular with diagonal entries $1$. Hence, if $M$ is the matrix of $[E_{m}],[F_{m}],[G_{m}]$ in the projective basis, then $1=\det(M)^{2}$, so $M$ is unimodular. Thus $[E_{m}],[F_{m}],[G_{m}]$ form a basis of $K_{0}(A_{m})$, and $K_{0}(\cP_{m})=0$ by additivity of $K_{0}(-)$. Similarly, $\HoH_{*}(\langle E_{m},F_{m},G_{m}\rangle)\cong k^{3}$, so additivity of $\HoH_{*}(-)$ gives $\HoH_{*}(\cP_{m})=0$. To see that $\cP_{m}$ is a phantom category, it remains to show $\cP_{m} \neq 0$.

We now compute the Hochschild cohomology.
Proposition~\ref{proposition-HH-regrading} gives
$$
\HoH^{d}(\cP_{m})
\cong
\bigoplus_{p+2mw=d}\HoH^{p,w}(\cP_{0}).
$$
Suppose first that $m>0$. Every class with negative internal degree has $p\geq2$. For
$-3\leq w<0$, we have 
$$
p+2mw\geq2-6m>2-8m.
$$
Hence the unique extreme bidegree $(2,-4)$ gives
$$
\min(j : \HoH^{j}(\cP_{m})\neq 0)=2-8m,
\qquad
\HoH^{2-8m}(\cP_{m})\cong k.
$$
Similarly, every class has $p\leq5$ and $w\leq3$, while
$\HoH^{5,3}(\cP_{0})\cong k^{2}$.  Therefore
$$
\max(j:\HoH^{j}(\cP_{m})\neq 0)=5+6m,
\qquad
\HoH^{5+6m}(\cP_{m})\cong k^{2}.
$$

For $m=0$, Lemma~\ref{lemma-HH-endpoints} gives vanishing outside $[0,5]$.
Moreover, $\HoH^{0,0}(\cP_{0})\cong k$ and there are no other weight components
in cohomological degree $0$, so
$\HoH^{0}(\cP_{0})\cong k$.
For $\HoH^{5}$, we have
$$
\dim_{k}\HoH^{5}(\cP_{0})
=f_{EF}(1)f_{FG}(1)v_{GE}(1)=5\cdot7\cdot4=140,
$$
so $\HoH^{5}(\cP_{0})\cong k^{140}$.

Finally, suppose that $m<0$. Since $0\leq p\leq5$ and $-4\leq w\leq3$, and
$\HoH^{0,w}(\cP_{0})=\HoH^{1,w}(\cP_{0})=0$ for $w\neq0$, every regraded class lies in
$[2+6m,5-8m]$. By Lemma~\ref{lemma-HH-endpoints},
$\HoH^{2,3}(\cP_{0})\cong k$ and $\HoH^{5,-4}(\cP_{0})\cong k$.
The bidegrees $(2,3)$ and $(5,-4)$ are the unique nonzero bidegrees mapping
to the 2 endpoints. Hence
$$
\HoH^{2+6m}(\cP_{m})\cong k,
\qquad
\HoH^{5-8m}(\cP_{m})\cong k.
$$
Finally, the ordered pairs of extremal nonzero Hochschild degrees are distinct for distinct $m$: first of all if the 2 indices have the same sign, the pairs are
$(2-8m,5+6m)$ for $m>0$, $(0,5)$ for $m=0$, and $(2+6m,5-8m)$ for $m<0$, hence they are distinct. If the 2 indices have different signs, then we have $8m=-6n$ and $6m=-8n$, hence $m=n=0$. So the categories $\cP_{m}$ are pairwise non-derived-Morita-equivalent.
\end{proof}

\section{Reproducibility}
The purpose of this section is to record how the explicit linear-algebra computations used above can be independently reproduced and checked.

The codes file accompanying the working version contains the
relation matrices, module matrices, internal degrees, expected rank
tables, the complete bigraded Hochschild table, and exact verification
routines.  All rank statements in the
main text are stated over $\QQ$ and may be checked by row
reduction over $\QQ$.  For speed and compact certification, the companion code also
performs the same rank checks modulo $101$ and includes the elementary
lemmas that lift those certificates back to characteristic 0. No
floating-point linear algebra is used.  The complete bigraded Hochschild table is retained in the computational
companion but is not needed for Theorem~\ref{theorem-main}.

\bibliographystyle{emss}
\bibliography{ref}

\end{document}